\documentclass[10pt]{article}
\usepackage{amsfonts}
\usepackage{amsthm}
\usepackage{amssymb}
\usepackage{graphicx, enumerate}
\usepackage{amsmath}
\usepackage{latexsym}
\usepackage{longtable}
\usepackage{tabularx}
\usepackage{amsmath}
\usepackage{amsfonts}
\usepackage{amsthm}
\usepackage{setspace}
\usepackage{graphicx}
\usepackage{float}
\usepackage{rotating}
\usepackage{tikz}
\usepackage{verbatim}

\newtheorem{theorem}{Theorem}
\newtheorem{thm}[theorem]{Theorem}
\newtheorem{oprb}[theorem]{Open Problem}
\newtheorem*{oprb*}{Open Problem}

\newtheorem{conj}[theorem]{Conjecture}

\theoremstyle{definition}
\newtheorem{definition}[theorem]{Definition}
\theoremstyle{definition}
\newtheorem*{construction*}{Construction}
\newtheorem{construction}[theorem]{Construction}

\def\lcm{\mathrm{lcm}}

\vspace{5cm}

\begin{document}

\textwidth4.5true in
\textheight7.2true in

\title
{\Large \sc \bf {Supermagic labeling of $C_n\Box C_m$}}
\date{}
\author{{{Dalibor Froncek, University of Minnesota Duluth}}\\
}
\maketitle

\begin{abstract}
A {supermagic labeling} (often also called supermagic labeling) of a graph $G(V,E)$ with $|E|=k$ is a bijection from $E$ to the set of first $k$ positive integers such that the sum of labels of all incident edges of every vertex $x\in V$ is equal to the same integer $c$. An existence of a supermagic labeling of Cartesian product of two cycles, $C_{n}\Box C_m$ for $n,m\geq4$ and both $n,m$ even and for any $C_n\Box C_n$ with $n\geq3$ was proved by 
Ivan\v{c}o~\cite{ivanco}. Ivan\v co also conjectured that such labeling is possible for any $C_n\Box C_m$ with $n,m\geq3$. We prove his conjecture for $n,m$ odd except when they are relatively prime.
\end{abstract}



\section{Introduction}

A {\em labeling} of a graph $G$ with vertex set $V$ of order $p$ and edge set $E$ of order $q$ is a bijection $f$ from $V, E,$ or $V\cup E$ to the set $\{1,2,\dots,s\}$ where $s=p,q$, or $p+q$, respectively, assigning to each element $a$ of $V, E,$ or $V\cup E$ its {\em label} $f(a)$.

The {\em weight} of an element is defined as the sum of labels of all adjacent or incident elements (or both), sometimes including the label of the element itself.

A {\em magic type} labeling is one in which the  weight of every vertex or edge of $G$ is equal to the same integer $c$, called the {\em magic constant}. Recently many authors studied also labelings using Abelian groups rather than sets of consecutive integers. Such labelings are then usually called {\em $\Gamma$-labelings}.

Probably the longest studied magic type labeling is the {\emph{supermagic labeling}}, also called the {\em vertex-magic edge labeling}. To avoid confusion with other labelings that have also been called supermagic by some authors, we use the latter term.

\begin{definition}\label{def:VMEL}
	Let $G$ be a graph with vertex set $V$ and edge set $E$ of order $q$, and $f$ a bijection from $E$ to the set of integers $\{1,2,\dots, q\}$, called {\em labels}. Define the {\em weight} $w(x)$ of a vertex $x\in V$ as the sum of labels of all edges incident with $x$. It the weight os every vertex is the same, that is, there is a positive integer $c$ called the {\em magic constant} such that 
	$$
	w(x)=\sum_{y:xy\in E} f(xy)=c
	$$
	for every $x\in V$, then the mapping $f$ is called a {\em supermagic labeling} or {\em supermagic labeling} of graph $G$.
\end{definition}

There are too many papers studying the supermagic labelings of graphs to be listed here. For a comprehensive overview, we refer the interested reader to Gallian's survey~\cite{Gal}. Vertex-transitive graphs are some of most interesting classes in this context. Ivan\v{c}o studied {Cartesian products} of two cycles.

\begin{definition}\label{def:cart-product}
	The {\em Cartesian  product} $G=G_{1}\Box G_{2}$ of graphs $G_{1}$ and $G_{2}$ with disjoint vertex and edge sets $V_{1}$, $V_{2}$, and $E_{1}$, $E_{2}$ respectively, is the graph with vertex set $V=V_{1} \times V_{2}$ where any two vertices $u=(u_{1},u_{2})\in G$ and $v=(v_{1},v_{2})\in G$ are adjacent in $G$ if and only if either $u_{1}=v_{1}$ and $u_{2}$ is adjacent with $v_{2}$ in $G_{2}$ or, $u_{2}=v_{2}$ and $u_{1}$ is adjacent with $v_{1}$ in $G_{1}$.
\end{definition}

In~\cite{ivanco}, Ivan\v co proved the following results.

\begin{theorem}\label{thm:ivanco-C_n-C_n}
	The Cartesian product $C_{n} \Box C_{n}$ has a supermagic labeling for any $n\geq 3$.
\end{theorem}

\begin{theorem}\label{thm:ivanco-C_2m-C_2n}
	Let $n,m \geq 2$ be integers. Then the Cartesian product $C_{2n} \Box C_{2m}$ has a supermagic labeling.
\end{theorem}
Ivan\v co also conjectured that the Cartesian product $C_{n} \Box C_{m}$ allows a supermagic labeling for any $n,m\geq 3$.

\begin{conj}\label{conj:ivanco}{\em (Ivan\v co)}
	The Cartesian product $C_{n} \Box C_{m}$ has a supermagic labeling for any $n,m\geq 3$.
\end{conj}

It turns out that  when the labeling is performed in a cyclic group $Z_{2nm}$ rather than in positive integers, than the conjecture is true. 

\begin{definition}\label{def:group-edge-laleling}
	A \emph{supermagic $\Gamma$-labeling} of a graph $G(V,E)$ with $|E|=k$ is a bijection $f$ from $E$ to an Abelian group $\Gamma$ of order $k$ such that the sum of labels of all incident edges of every vertex $x\in V$, called the \emph{weight} of $x$ and denoted $w(x)$, is equal to the same element $\mu \in \Gamma$, called the \emph{magic constant}. That is,
	$$
	w(x) =\sum_{y:xy\in E} f(xy) =\mu
	$$
	for every vertex $x\in V$.
\end{definition}

Froncek, McKeown, McKeown, and McKeown in~\cite{F-McK^3} proved the following.

\begin{theorem}\label{thm:Gamma}
	The Cartesian product $C_n \Box C_m$ admits a supermagic $Z_{2nm}$-labeling 
	for all $n,m\geq 3$.
\end{theorem} 

We will show that Ivan\v co's conjecture is true for cycles of odd lengths that are not relatively prime by proving the following.

\begin{thm}\label{thm:main}
	Let $n,m$ be odd positive integers with $\gcd(n,m)>1$. Then there exists a supermagic labeling of the Cartesian product $C_n\Box C_m$ with the magic constant $c=4nm+2$.
\end{thm}

The construction is based on a similar one used for $Z_{2nm}$-labeling of $C_n\Box C_m$ in~\cite{F-MMcK} by Froncek and McKeown.

Since Ivan\v co's proof in~\cite{ivanco} was strictly existential and did not provide a construction of a particular labeling, we also present a labeling for the case covered by Theorem~\ref{thm:ivanco-C_2m-C_2n}.

For completeness we also mention that analogous results for {\em distance magic} and {\em $\Gamma$-distance magic} labeling of $C_n\Box C_m$ (where we label vertices rather than edges and the weight of a vertex $x$ is the sum of labels of all neighbors of $x$) were proved in~\cite{Rao}, \cite{fro2}, and~\cite{cichacz1}.
\section{Construction}\label{sec:D-construction}

We  use the following notation. Vertices will be denoted $x_{ij}$ for $i=1,2,\dots,n$ and $j=1,2,\dots,m$. Every vertex $x_{ij}$ is then incident with two \emph{vertical edges}  $x_{(i-1)j}x_{ij}$ and $x_{ij}x_{(i+1)j}$ and two \emph{horizontal edges}  $x_{i(j-1)}x_{ij}$ and $x_{ij}x_{i(j+1)}$ where the subscripts are computed modulo $n$ and $m$, respectively.

By a \emph{diagonal} $D^j$ we mean a cycle consisting alternately of horizontal and vertical edges containing the horizontal edge $x_{1j}x_{1(j+1)}$. 
%
More precisely, 
$$
D^{j}=(x_{1j}x_{1(j+1)},x_{1(j+1)}x_{2(j+1)},x_{2(j+1)}x_{2(j+2)},x_{2(j+2)}x_{3(j+2)},\dots,x_{nj}x_{1j}).
$$ 

We observe that each diagonal $D^j$ is a cycle of length $2l$, where $l=\lcm(n,m)$. To see that, we notice that to get back to $x_{1j}$ we need to pass through $am$ horizontal edges and $bn$ vertical edges, where $a,b$ are positive integers. Because the number of horizontal and vertical edges in the diagonal is the same, we must have $am=bn=\lcm(n,m)$ and the conclusion follows. Since each diagonal has $2l$ edges, the number of diagonals, call it $d$, is $d=2nm/2l=nm/l=\gcd(n,m)$.

To avoid complicated notation in vertex subscripts, we just denote the horizontal edges in $D^j$ consecutively as $h^j_1,h^j_2,\dots,h^j_{l}$ and the vertical ones by $v^j_1,v^j_2,\dots,v^j_{l}$. All diagonals will start at vertex $x_{1j}$, that is, 
$h^j_1=x_{1j}x_{1(j+1)},v^j_1=x_{1(j+1)}x_{2(j+1)}$ and so on. 
We will also call the pair of consecutive edges $(h^j_k,v^j_k)$ the \emph{$HV$-corner} and the pair  $(v^j_{k-1},h^j_{k})$ the \emph{VH-corner}. In particular, the pairs will be called the $k$-th $HV$-corner and $k$-th $VH$-corner, respectively. Note that the first $HV$-corner is $(h^j_1,v^j_1)$ while the first $VH$-corner is $(v^j_l,h^j_1)$.

\begin{construction}\label{const:odd-odd}
We construct a labeling $f$ of $C_n\Box C_m$ for $3\leq n\leq m$, where $n,m$ are odd and $\gcd(n,m)>1$.
\end{construction} 

\noindent
Since $n,m$ are odd, we have both $l$ and $d$ odd and because we have assumed that $n$ and $m$ are co-prime, we have $d\geq3$. We set $l=2l'+1$.

We label each diagonal $D^j$ for $j=1,2,\dots,d-1$ alternately with two sets of $l$ consecutive integers. The horizontal edges will be labeled in increasing order and the vertical ones in decreasing order. The last diagonal $D^d$ will be labeled differently.

All diagonals except  $D^1$ will start at vertex $x_{1j}$. Hence, we have 
$h^j_1=x_{1j}x_{1(j+1)},v^j_1=x_{1(j+1)}x_{2(j+1)}$ and so on. For $D^1$, we start at vertex $x_{1(d+1)}$, obtaining $h^1_1=x_{1(d+1)}x_{1(d+2)},v^1_1=x_{1(d+2)}x_{2(d+2)}$ and so on.
Recall that the first $HV$-corner is $(h^j_1,v^j_1)$ while the first $VH$-corner is $(v^j_l,h^j_1)$.

The odd diagonals except the last one will be labeled as follows.

\begin{equation*}\label{eq:D^j-odd}
 D^j \ \text{for}\ j\ \text{odd}, \ j<d\ 
\begin{cases}
f(h^j_k)= (j-1)l + k\\
\\
f(v^j_k)= 2nm - (j-1)l -k +1\\
\end{cases}
\end{equation*}

 
Now we look at the partial weights created by the labels in each $D^j$. At the $k$-th $HV$-corner, we have the partial weight $w^j_{HV}$ of the appropriate vertex of $D^j$ equal to 
\begin{align}\label{eq:D^j-odd-HV-k}
	w^j_{HV}(x_{st})	&= f(h^j_k)+f(v^j_k)\nonumber\\
					&= ((j-1)l + k) + (2nm - (j-1)l -k +1) = 2nm+1.
\end{align}
At the first $VH$ corner for $j=1$ we have
\begin{align}\label{eq:D^j-odd-VH-first-j=1}
w^1_{VH}(x_{1(d+1)})&= f(v^1_l) + f(h^1_{1})\nonumber\\
					&= (2nm - l +1) + 1 = 2nm-l+2,
\end{align}
and for $3\leq j\leq d-2$ we have
\begin{align}\label{eq:D^j-odd-VH-first}
w^j_{VH}(x_{1j})		&= f(v^j_l) + f(h^j_{1})\nonumber\\
&= (2nm - (j-1)l -l +1) + ((j-1)l + 1) = 2nm-l+2
\end{align}
as well.
 At the $k$-th $VH$-corner for $2\leq k\leq l$ the partial weight $w^j_{VH}$ is
\begin{align}\label{eq:D^j-odd-VH-k>1}
w^j_{VH}(x_{st})	&= f(v^j_{k-1}) + f(h^j_{k})\nonumber\\
					&= (2nm - (j-1)l -(k-1) +1) + ((j-1)l + k) = 2nm+2,
\end{align}
regardless of the value of $j$. 

Now we want to label the even diagonals so that their $VH$-corners would have partial weights equal to $2nm+1$, the weights of $HV$-corners would be $2nm$ or $2nm+l$ (at one ``exceptional'' corner) and the vertices with exceptional weights would be aligned. To do so, we need to shift the horizontal edge labels by one position down, while labeling the vertical edges as before. Namely, for $j<d-1$ we have
\begin{equation*}\label{eq:D^j-even}
D^j \ \text{for}\ j\ \text{even} \
\begin{cases}
f(h^j_1)= (j-1)l + l=jl\\
f(h^j_k)= (j-1)l + k-1 \ \text{for}\ k>1\\
f(v^j_k)= 2nm - (j-1)l -k +1\\
\end{cases}
\end{equation*}
The partial weights at the first $HV$-corner are  then 
\begin{align}\label{eq:D^j-even-HV-first}
w^j_{HV}(x_{1(j+1)})	&= f(h^j_1)+f(v^j_1)\nonumber\\
				&= ((j-1)l + l) + (2nm - (j-1)l) = 2nm+l
\end{align}
and at the $k$-th $HV$-corner for $k>1$ we have
\begin{align}\label{eq:D^j-even-HV-k>1}
w^j_{HV}(x_{st})	&= f(h^j_k)+f(v^j_k)\nonumber\\
				&= ((j-1)l + k-1) + (2nm - (j-1)l -k +1) = 2nm.
\end{align}
At the $k$-th $VH$-corner the partial weight $w^j_{VH}$ is
\begin{align}\label{eq:D^j-even-VH-k}
w^j_{VH}(x_{st})	&= f(v^j_{k-1}) + f(h^j_{k})\nonumber\\
					&= (2nm - (j-1)l -(k-1) +1) + ((j-1)l + k-1) = 2nm+1
\end{align}
for any value of $j$. The last even diagonal $D^{d-1}$ is labeled in a similar way except that the vertex with exceptional partial weight is shifted. Namely, we have
\begin{equation*}\label{eq:D^j-(d-1)}
%
D^{d-1} \ 
\begin{cases}
f(h^{d-1}_k)= (d-2)l + k+l'-1 \ \text{for}\ k\leq l'+2\\
f(h^{d-1}_k)= (d-2)l + k-l'-2 \ \text{for}\ k\geq l'+3\\
f(v^{d-1}_k)= 2nm - (d-2)l -k-l' +1 \ \text{for}\ k\leq l'+1\\
f(v^{d-1}_k)= 2nm - (d-2)l -k+l'+2 \ \text{for}\ k\geq l'+2\\
\end{cases}
\end{equation*}
The partial weights at the first $l'+1$ $HV$-corners are   
\begin{align}\label{eq:D^{d-1}-HV-first}
w^{d-1}_{HV}(x_{st})	&= f(h^{d-1}_k)+f(v^{d-1}_k)\nonumber\\
&= ((d-2)l+k+l'-1) + (2nm-(d-2)l -k-l'+1) = \nonumber\\
&= 2nm
\end{align}
and at the $(l'+2)$-nd $HV$-corner we get the exceptional weight 
\begin{align}\label{eq:D^{d-1}-HV-k=l'+2}
w^{d-1}_{HV}(x_{st})	&= f(h^{d-1}_{l'+2})+f(v^{d-1}_{l'+2})\nonumber\\
&= ((d-2)l+(l'+2) +l'-1) + (2nm-(d-2)l -(l'+2)+l'+2) = \nonumber\\
&= ((d-2)l+l) + (2nm-(d-2)l) = \nonumber\\
&= 2nm+l.
\end{align}
At the remaining $HV$-corners for $k=l'+3,l'+4,\dots,l$ the weights are again
\begin{align}\label{eq:D^{d-1}-HV-last}
w^{d-1}_{HV}(x_{st})	&= f(h^{d-1}_{l'+2})+f(v^{d-1}_{l'+2})\nonumber\\
&= ((d-2)l+k-l'-2) + (2nm-(d-2)l -k+l'+2) = \nonumber\\
&= 2nm.
\end{align}

At the $VH$-corners, the partial weights $w^{d-1}_{VH}$ are as follows. 
For $k=1$ we have
\begin{align}\label{eq:D^{d-1}-VH-k=1}
w^{d-1}_{VH}(x_{1(d-1)})	&= f(v^{d-1}_{l}) + f(h^{d-1}_{1})\nonumber\\
&=  (2nm - (d-2)l -l+l'+2) +((d-2)l + 1+l'-1) \nonumber\\
&= 2nm-l+2l'+2 =2nm+1
\end{align}
and for $k=2,3,\dots,l'+2$ we have
\begin{align}\label{eq:D^{d-1}-VH-first}
w^{d-1}_{VH}(x_{st})	&= f(v^{d-1}_{k-1}) + f(h^{d-1}_{k})\nonumber\\
&= (2nm - (d-2)l -(k-1)-l' +1) + ((d-2)l + k+l'-1) \nonumber\\
&= 2nm+1.
\end{align}
For $k=l'+3,l'+4,\dots,l$,
\begin{align}\label{eq:D^{d-1}-VH-last}
w^{d-1}_{VH}(x_{st})&= f(v^{d-1}_{k-1}) + f(h^{d-1}_{k})\nonumber\\
					&=  (2nm - (d-2)l -(k-1)+l'+2) +((d-2)l + k-l'-2) \nonumber\\
					&= 2nm+1
\end{align}
as in all previous cases.

The remaining odd diagonal, $D^d$, is labeled in a different way. We define
\begin{equation*}\label{eq:D^j-d}
%
D^{d} \ 
\begin{cases}
f(h^{d}_1)= dl \\
f(h^{d}_k)= (d-1)l + 2k-2 \ \text{for}\ 2\leq k\leq l'+1\\
f(h^{d}_k)= (d-2)l + 2k-2 \ \text{for}\ k\geq l'+2\\
f(v^{d}_k)= 2nm - (d-1)l -2k+2 \ \text{for}\ k\leq l'+1\\
f(v^{d}_k)= 2nm - (d-2)l - 2k+2 \ \text{for}\ k\geq l'+2\\
\end{cases} 
\end{equation*}
The partial weight at the first  $HV$-corner is 
\begin{align}\label{eq:D^d-HV-first}
w^{d}_{HV}(x_{1(d+1)})	&= f(h^{d}_1)+f(v^{d}_1)\nonumber\\
&= dl + (2nm - (d-1)l -2+2) = \nonumber\\
&= 2nm +l, 
\end{align}
and at the $k$-th $HV$-corner for $k=2,3,\dots, l'+1$ as 
\begin{align}\label{eq:D^d-HV-k<l'+2}
w^{d}_{HV}(x_{st})	&= f(h^{d}_{k})+f(v^{d}_{k})\nonumber\\
&= ((d-1)l + 2k-2) + (2nm - (d-1)l -2k+2) = \nonumber\\
&= 2nm.
\end{align}
At the remaining $HV$-corners for $k=l'+2,l'+3,\dots,l$ the weights are  
\begin{align}\label{eq:D^d-HV-last}
w^{d}_{HV}(x_{st})	&= f(h^{d}_{k})+f(v^{d}_{k})\nonumber\\
&= ((d-2)l + 2k-2) + (2nm - (d-2)l -2k+2) = \nonumber\\
&= 2nm
\end{align}
as well.
At the $VH$-corners, the partial weights $w^{d}_{VH}$ are as follows. 
For $k=1$ we have
\begin{align}\label{eq:D^d-VH-k=1}
w^{d}_{VH}(x_{1d})	&= f(v^j_{l}) + f(h^j_{1})\nonumber\\
&=  (2nm - (d-2)l -2l+2) +dl \nonumber\\
&= 2nm+2
\end{align}
and for $k=2,3,\dots,l'+1$ we have
\begin{align}\label{eq:D^d-VH-first}
w^{d}_{VH}(x_{st})	&= f(v^d_{k-1}) + f(h^d_{k})\nonumber\\
&= (2nm - (d-1)l -2(k-1)+2) + ((d-1)l + 2k-2) \nonumber\\
&= 2nm+2
\end{align}
as well. For $k=l'+2$ we obtain the exceptional weight
\begin{align}\label{eq:D^d-VH-k=l'+2}
w^{d}_{VH}(x_{st})	&= f(v^d_{l'+1}) + f(h^d_{l'+2})\nonumber\\
&= (2nm - (d-1)l -2(l'+1)+2) + ((d-2)l + 2(l'+2)-2) \nonumber\\
&= (2nm - (d-1)l -2l') + ((d-2)l + 2l'+2) \nonumber\\
&= 2nm-l+2
\end{align}
and for $k=l'+3,l'+4,\dots,l$,
\begin{align}\label{eq:D^d-VH-last}
w^{d}_{VH}(x_{st})	&= f(v^d_{k-1}) + f(h^d_{k})\nonumber\\
&=  (2nm - (d-2)l - 2(k-1)+2) +((d-2)l + 2k-2) \nonumber\\
&= 2nm+2
\end{align}
as in all previous cases except $k=l'+2$.
\qed

\vskip10pt

The construction for $n,m$ even is based on the same idea and is actually much simpler.

\begin{construction}\label{const:even-even}
	We construct a labeling $f$ of $C_n\Box C_m$ for $4\leq n\leq m$, where $n,m$ are even.
\end{construction} 

Because we have $n,m$ both even, the number of diagonals $d=\gcd(n,m)$ is also even. This fact simplifies matters a lot. 
The construction is very similar to the previous one. The only difference is that because we have an even number of diagonals, we will not need the exceptional one.

We label all odd diagonals as follows, which is the same as the first odd diagonals in Construction~\ref{const:odd-odd}. The only difference is that we start each diagonal $D^j$ at $x_{1j}$, including $D^1$, which will then of course start at $x_{11}$.

\begin{equation*}\label{eq:even_D^j-odd}
	D^j \ \text{for}\ j\ \text{odd}, \ j<d\ 
	\begin{cases}
		f(h^j_k)= (j-1)l + k\\
		\\
		f(v^j_k)= 2nm - (j-1)l -k +1\\
	\end{cases}
\end{equation*}

The partial weights at the  $HV$-corners are by~(\ref{eq:D^j-odd-HV-k}) again equal to  
\begin{align}\label{eq:even_D^j-odd-HV-k}
	w^j_{HV}(x_{st})	&=  2nm+1.
\end{align}
At the first $VH$ corner for we have 
\begin{align}\label{eq:even_D^j-odd-VH-first}
w^j_{VH}(x_{1j})		&= f(v^j_l) + f(h^j_{1})\nonumber\\
&= (2nm - (j-1)l -l +1) + ((j-1)l + 1) = 2nm-l+2
\end{align}
for every odd $j$.
At the $k$-th $VH$-corner for $2\leq k\leq l$ the partial weight $w^j_{VH}$ is by~(\ref{eq:D^j-odd-VH-first}) 
\begin{align}\label{eq:even_D^j-odd-VH-k>1}
	w^j_{VH}(x_{st})	&=  2nm+2
\end{align}
for every odd  $j$. 

We again label the even diagonals so that their $VH$-corners  have partial weights equal to $2nm+1$, the weights of $HV$-corners are  $2nm$ or $2nm+l$ (at the ``exceptional'' corner) and the vertices with exceptional weights match. 
\begin{equation*}\label{eq:even_D^j-even}
	D^j \ \text{for}\ j\ \text{even} \
	\begin{cases}
		f(h^j_1)= (j-1)l + l=jl\\
		f(h^j_k)= (j-1)l + k-1 \ \text{for}\ k>1\\
		f(v^j_k)= 2nm - (j-1)l -k +1\\
	\end{cases}
\end{equation*}
The partial weights at the first $HV$-corner are  then~(\ref{eq:D^j-even-HV-first})
\begin{align}\label{eq:even_D^j-even-HV-first}
	w^j_{HV}(x_{1(j+1)})	&=  2nm+l
\end{align}
as in Observation~\ref{const:odd-odd}  and similarly at the $k$-th $HV$-corner for $k>1$ we have by~(\ref{eq:D^j-even-HV-k>1})
\begin{align}\label{eq:even_D^j-even-HV-k>1}
	w^j_{HV}(x_{st})	&=  2nm.
\end{align}
At the $k$-th $VH$-corner the partial weight $w^j_{VH}$ is
\begin{align}\label{eq:even_D^j-even-VH-k}
	w^j_{VH}(x_{st})	&=  2nm+1
\end{align}
for any even value of $j$, as follows from~(\ref{eq:D^j-even-VH-k}). 
\qed
\section{Main Result}\label{sec:proof}

Now we are ready to prove our result.

\begin{thm}\label{thm:main-odd-odd}
	Let $n,m$ be positive integers of the same parity with $\gcd(n,m)>1$. Then the labeling $f$ described in {Constructions~\ref{const:odd-odd}} and~{\ref{const:even-even}} is a supermagic labeling of the Cartesian product $C_n\Box C_m$ with the magic constant $c=4nm+2$.
\end{thm}

\begin{proof}	
First we observe that every vertex $x_{st}$ of $C_n\Box C_m$ belongs to two consecutive diagonals, say $D^j$ and $D^{j+1}$. Moreover, the edges 
$x_{s(t-1)}x_{st}$ and $x_{st}x_{(s+1)t}$ form a $k$-th $HV$-corner in $D^j$ for some $k$ while
the edges 
$x_{(s-1)t}x_{st}$ and $x_{st}x_{s(t+1)}$ form a $k$-th $VH$-corner in $D^{j+1}$.
It follows that the weight $w(x_{st})$ is the sum of the partial weights, that is,
\begin{equation*}
	w(x_{st}) = w^j_{HV}(x_{st}) + w^{j+1}_{VH}(x_{st}). 
\end{equation*}	
First we look at the case of $n,m$ odd. 
In Construction~\ref{const:odd-odd}, when $j$ is odd and $j\leq d-2$, it follows from~(\ref{eq:D^j-odd-HV-k})	that $w^j_{HV}(x_{st})$ in diagonal $D^j$ is always equal to $2nm+1$. Similarly, it follows from~(\ref{eq:D^j-even-VH-k}) for 
$2\leq j+1\leq d-3$ and from~(\ref{eq:D^{d-1}-VH-k=1}),~(\ref{eq:D^{d-1}-VH-first}), 
and~(\ref{eq:D^{d-1}-VH-last}) for $j+1=d-1$ that $w^{j+1}_{VH}(x_{st})$ is always equal to $2nm+1$ as well. Therefore, for $j$ odd, $1\leq j\leq d-2$ we have
\begin{align*}
w(x_{st}) 	&= w^j_{HV}(x_{st}) + w^{j+1}_{VH}(x_{st})\nonumber\\
			&=(2nm+1)+(2nm+1)=4nm+2. 
\end{align*}		
For the remaining odd value of $j=d$ we have diagonal $D^d$ and the following diagonal is $D^1$. At the first $HV$-corner of $D^d$ we have from~(\ref{eq:D^d-HV-first}) $w^{d}_{HV}(x_{1(d+1)})=2nm+l$ and from~(\ref{eq:D^j-odd-VH-first}) we have $w^1_{VH}(x_{1(d+1)})=2nm-l+2$ (recall that diagonal $D^1$ starts at vertex $x_{1(d+1)}$). Hence,
\begin{align*}
w(x_{1(d+1)}) 	&= w^d_{HV}(x_{1(d+1)}) + w^{1}_{VH}(x_{1(d+1)})\nonumber\\
				&=(2nm+l)+(2nm-l+2)=4nm+2. 
\end{align*}		
At the remaining $HV$-corners for $k=2,3,\dots,l$ we have $w^{d}_{HV}(x_{rs})=2nm$  from~(\ref{eq:D^d-HV-k<l'+2}) and~(\ref{eq:D^d-HV-last}) and $w^1_{VH}(x_{rs})=2nm+2$ 
from~(\ref{eq:D^j-odd-VH-k>1}). Hence,
\begin{align*}
w(x_{rs}) 	&= w^d_{HV}(x_{rs}) + w^{1}_{VH}(x_{rs})\nonumber\\
&=(2nm)+(2nm+2)=4nm+2. 
\end{align*}		

Next we look at weights of vertices that belong to consecutive diagonals $D^j$ and $D^{j+1}$ for $j$ even, $2\leq j\leq d-3$, provided $d>3$ and such diagonals exist. The partial weight at the first $HV$-corner is $w^j_{HV}(x_{1(j+1)})=2nm+l$, as follows from~(\ref{eq:D^j-even-HV-first}), while from~(\ref{eq:D^j-odd-VH-first}) we have  $w^j_{VH}(x_{1(j+1)})=2nm-l+2$.	
Therefore,	
\begin{align*}
w(x_{1(j+1)}) 	&= w^j_{HV}(x_{1(j+1)}) + w^{j+1}_{VH}(x_{1(j+1)})\nonumber\\
				&=(2nm+l)+(2nm-l+2)=4nm+2. 
\end{align*}		
For the remaining vertices from~(\ref{eq:D^j-even-HV-k>1}) we have $w^j_{HV}(x_{st})=2nm$  and from~(\ref{eq:D^j-odd-VH-k>1}) $w^{j+1}_{VH}(x_{1(j+1)})=2nm+2$, and
\begin{align*}
w(x_{st}) 	&= w^j_{HV}(x_{st}) + w^{j+1}_{VH}(x_{st})\nonumber\\
			&=(2nm)+(2nm+2)=4nm+2. 
\end{align*}		
Finally, we examine weights of vertices that belong to diagonals $D^{d-1}$ and $D^{d}$. 
For the $(l'+2)$-nd $HV$-corner, we have $w^{d-1}_{HV}(x_{st})=2nm+l$ from~(\ref{eq:D^{d-1}-HV-k=l'+2}) and $w^{d}_{VH}(x_{st})=2nm-l+2$	from~(\ref{eq:D^d-VH-k=l'+2}), and
\begin{align*}
w(x_{st}) 	&= w^{d-1}_{HV}(x_{st}) + w^{d}_{VH}(x_{st})\nonumber\\
			&=(2nm+l)+(2nm-l+2)=4nm+2. 
\end{align*}		
For the remaining vertices we have at $HV$-corners $w^{d-1}_{HV}(x_{st})=2nm$, as follows from~(\ref{eq:D^{d-1}-HV-first}) and~(\ref{eq:D^{d-1}-HV-last}), while from~(\ref{eq:D^d-VH-k=1}),~(\ref{eq:D^d-VH-first}) and~(\ref{eq:D^d-VH-last}) we have  $w^d_{VH}(x_{st)})=2nm+2$.	
This yields	
\begin{align*}
w(x_{st}) 	&= w^{d-1}_{HV}(x_{st}) + w^{d}_{VH}(x_{st})\nonumber\\
			&=(2nm)+(2nm+2)=4nm+2. 
\end{align*}		

This concludes the case of $n,m$ both odd, as the weight s of all vertices have been verified.

The case of $n,m$ even is simpler. In Construction~\ref{const:even-even}
when $j$ is odd, it follows from~(\ref{eq:even_D^j-odd-HV-k})	that $w^j_{HV}(x_{st})$ in diagonal $D^j$ is always equal to $2nm+1$. Similarly, it follows from~(\ref{eq:even_D^j-even-VH-k}) for 
$2\leq j+1\leq d$ that $w^{j+1}_{VH}(x_{st})$ is always equal to $2nm+1$ as well. Therefore, for $j$ odd, $1\leq j\leq d-1$ we have
\begin{align*}
w(x_{st}) 	&= w^j_{HV}(x_{st}) + w^{j+1}_{VH}(x_{st})\nonumber\\
&=(2nm+1)+(2nm+1)=4nm+2. 
\end{align*}

Finally, we verify the weights of vertices at the seams between $D^j$ and $D^{j+1}$ for $j$ even. This also includes the case when $j=d$ and $j+1=1$.

From~(\ref{eq:even_D^j-even-HV-first}) it follows that the  partial weight at the first $HV$-corner is given as $w^j_{HV}(x_{1(j+1)})=2nm+l$ and 
from~(\ref{eq:even_D^j-odd-VH-first}) we have  the corresponding partial weight $w^j_{VH}(x_{1(j+1)})=2nm-l+2$.	
Therefore,	
\begin{align*}
w(x_{1(j+1)}) 	&= w^j_{HV}(x_{1(j+1)}) + w^{j+1}_{VH}(x_{1(j+1)})\nonumber\\
&=(2nm+l)+(2nm-l+2)=4nm+2. 
\end{align*}		
For the remaining vertices we have $w^j_{HV}(x_{st})=2nm$ from~(\ref{eq:even_D^j-even-HV-k>1}) and from~(\ref{eq:even_D^j-odd-VH-k>1}) we have $w^{j+1}_{VH}(x_{1(j+1)})=2nm+2$. Hence,
\begin{align*}
w(x_{st}) 	&= w^j_{HV}(x_{st}) + w^{j+1}_{VH}(x_{st})\nonumber\\
&=(2nm)+(2nm+2)=4nm+2. 
\end{align*}		
This concludes the case of $n,m$ both even.

We have now  computed weights of all vertices of our graph and concluded that $w(x_{st})=4nm+2$ for every vertex $x_{st}$, which completes the proof.	
\end{proof}

Our result then reduces the original conjecture by Ivan\v{c}o to the following open problem.

\begin{oprb}
	Does there exist a supermagic labeling of the Cartesian product $C_n\Box C_m$ for $n$ odd and $m$ even 
	or for relatively prime odd numbers $n,m$?
\end{oprb}

We concur with Ivan\v co that the answer is affirmative. Our belief is supported by the existence of supermagic labelings of $C_3\Box C_4, C_3\Box C_5$, and $C_3\Box C_6$ found computationally by T. Michna.
~\cite{comput}.

\vskip1cm

\noindent





\end{document}